\documentclass[reqno, a4paper, 11pt]{amsart}

\usepackage{amssymb,amsmath,mathtools,amsthm,bbm}
\usepackage{enumerate}
\usepackage{xcolor}
\usepackage{mathrsfs}
\usepackage{enumitem}
\usepackage{color}
\usepackage{txfonts}

\usepackage{tikz}
\usepackage{pgfplots}
\usepgfplotslibrary{
	groupplots,
	external,
}
\usetikzlibrary{arrows.meta,backgrounds,patterns,decorations.pathreplacing,decorations.pathmorphing}

\AtBeginDocument{\mathcode`v=\varv}

\usepackage{hyperref}  % pour les hyperliens (table des matières notamment)
\hypersetup{                    % parametrage des hyperliens
	colorlinks=true,                % colorise les liens
	breaklinks=true,                % permet les retours à la ligne pour les liens trop longs
	urlcolor= blue,                 % couleur des hyperliens
	linkcolor= black,                % couleur des liens internes aux documents (index, figures, tableaux, equations,...)
	citecolor= magenta,                % couleur des liens vers les references bibliographiques
	pdfstartview = FitH,	bookmarksopen = true               % marques-page déroulés
}

\usepackage{txfonts}
\usepackage{newtxtext}
\usepackage{tikz}
\usepackage{caption}
\usepackage{subcaption}
\usepackage{accents}
\usepackage{cleveref}

\newcommand{\W}{\mathcal{W}}

\newcommand{\R}{\mathbb{R}}
\newcommand{\N}{\mathbb{N}}
\newcommand{\E}{\mathbb{E}}

\DeclareMathOperator{\co}{\operatorname{co}}

\renewcommand{\epsilon}{\varepsilon}

\numberwithin{equation}{section}
\theoremstyle{plain}
\newtheorem{prooff}{Proof}[section]

\newtheorem{lemma}[prooff]{Lemma}
\newtheorem{theorem}[prooff]{Theorem}
\newtheorem{proposition}[prooff]{Proposition}

\theoremstyle{definition}
\newtheorem{example}[prooff]{Example}

\begin{document}

\title[Lipschitz continuity of the Wasserstein projection]{Lipschitz continuity of the Wasserstein projections in the convex order on the line}

\author{Benjamin Jourdain}
\thanks{CERMICS, Ecole des Ponts, INRIA, Marne-la-Vallée, \href{mailto:benjamin.jourdain@enpc.fr}{benjamin.jourdain@enpc.fr}}
\author{William Margheriti}
\author{Gudmund Pammer}
\thanks{Department of Mathematics, ETH Zürich, \href{mailto:gudmund.pammer@math.ethz.ch}{gudmund.pammer@math.ethz.ch}.}

\maketitle
\begin{abstract}
	Wasserstein projections in the convex order were first considered in the framework of weak optimal transport, and found application in various problems such as concentration inequalities and martingale optimal transport.
	In dimension one, it is well-known that the set of probability measures with a given mean is a lattice w.r.t. the convex order.
	Our main result is that, contrary to the minimum and maximum in the convex order, the Wasserstein projections are Lipschitz continuity w.r.t.\ the Wasserstein distance in dimension one.
	Moreover, we provide examples that show sharpness of the obtained bounds for the Wasserstein distance with index $1$.
\end{abstract}


\section{Introduction and main result}
Motivated by the restauration of the convex order between discrete approximations of two probability measures on $\R^d$, Alfonsi, Corbetta and one of the authors \cite{AlCoJo20} introduced for $p\in[1,+\infty)$ the Wasserstein projections in the convex order:
\begin{align}
	\label{eq:def.Wasserstein_projection.I}
	\mathcal I_p(\mu,\nu) :=& \arg \min \left\{ \W_p(\mu,\eta) \colon \eta \le_c \nu \right\}, \\
	\label{eq:def.Wasserstein_projection.J}
	\mathcal J_p(\mu,\nu) :=& \arg \min \left\{ \W_p(\eta,\nu) \colon \mu \le_c \eta \right\}.
\end{align}
Here $\W_p$ denotes the celebrated Wasserstein distance on the set $\mathcal P_p(\R^d)$ of Borel probability measures on $\R^d$ with finite $p$-th moment:
\begin{equation}
	\label{eq:def.Wasserstein_distance}
	\textstyle
	\W_p^p (\mu,\nu) := \inf_{ \pi \in \Pi(\mu,\nu)} \int_{\R^d \times \R^d} |x - y|^p \pi(dx,dy)\quad
	\mu, \nu \in \mathcal P_p(\R^d),
\end{equation}
where we write $\Pi(\mu,\nu)$ for the subset of couplings $\pi \in \mathcal P_p(\R^d \times \R^d)$ with first marginal $\mu$ and second marginal $\nu$. Moreover, the inequality $\mu \le_c \nu$ means that $\mu$ is smaller than $\nu$ in the convex order, i.e., for all convex $f \colon \R^d \to \R$
\begin{equation}
	\label{eq:def.convex_order}
	\textstyle
	\int_{\R^d} f(x) \, \mu(dx) \le \int_{\R^d} f(x) \, \nu(dy).
\end{equation}
When $d = 1$, the projections \eqref{eq:def.Wasserstein_projection.I} and \eqref{eq:def.Wasserstein_projection.J} enjoy additional properties according to \cite{AlCoJo20}: the probability measures $\mathcal I(\mu,\nu)$ resp.\ $\mathcal J(\mu,\nu)$ defined in \eqref{eq:quantileProjections} below are (when $p>1$ unique) optimizers of \eqref{eq:def.Wasserstein_projection.I} resp.\ \eqref{eq:def.Wasserstein_projection.J} for all $p\ge 1$. Our main result is the following Lipschitz continuity property of $\mathcal I$ and $\mathcal J$.
\begin{theorem}[Lipschitz continuity]\label{thm:main_result}
	When $d=1$ and $p \in [1,\infty)$, the Wasserstein projections $\mathcal I$, $\mathcal J$ are Lipschitz continuous.
	For  $\mu,\nu,\mu',\nu'\in\mathcal P_p(\R)$, we have
	\begin{align}
		\label{eq:I_Lipschitz}
		\mathcal W_p(\mathcal I(\mu,\nu),\mathcal I(\mu',\nu'))&\le 2\mathcal W_p(\mu,\mu')+\phantom{2}\mathcal W_p(\nu,\nu'),\\
		\label{eq:J_Lipschitz}
		\mathcal W_p(\mathcal J(\mu,\nu),\mathcal J(\mu',\nu'))&\le\phantom{2}\mathcal W_p(\mu,\mu')+2\mathcal W_p(\nu,\nu').
	\end{align}
\end{theorem}% 
Theorem \ref{thm:main_result} generalises \cite[Proposition 3.1 and Proposition 4.3]{AlCoJo20} when $d = 1$.
These propositions state that for any dimension $d \in \N$ and probabilities $\mu,\nu,\mu',\nu'\in\mathcal P_p(\R^d)$ with $\mu\le_c\nu$, so that $\mathcal I_p(\mu,\nu)=\mu$ and $\mathcal J_p(\mu,\nu)=\nu$, we have
\begin{align*}
	\mathcal W_p(\mathcal I_p(\mu,\nu),\mathcal I_p(\mu',\nu'))&\le 2\mathcal W_p(\mu,\mu')+\phantom{2}\mathcal W_p(\nu,\nu'),\\
	\mathcal W_p(\mathcal J_p(\mu,\nu),\mathcal J_p(\mu',\nu'))&\le\phantom{2}\mathcal W_p(\mu,\mu')+2\mathcal W_p(\nu,\nu').
\end{align*}
Hence, by Theorem \ref{thm:main_result} it is possible (for $d=1$) to drop the convex ordering constraint $\mu\le_c\nu$.
The extension of Theorem \ref{thm:main_result} to dimensions $d > 1$ is to the authors' understanding an interesting open question.

\subsection{Discussion on Wasserstein projections and related problems}
Gozlan, Roberto, Samson, and Tetali \cite{GoRoSaTe17} introduced a generalization of optimal transport, the weak optimal transport, in order to study measure concentration inequalities.
The following barycentric weak optimal transport problem received in recent years special attention, see for example \cite{GoRoSaTe17,GoRoSaSh18,GoJu18,AlCoJo19,AlCoJo20,BaBePa18,BaBePa19}:
for $\mu, \nu \in \mathcal P_p(\R^d)$, consider
\begin{equation}\label{eq:def_WOT}
	\textstyle
	\mathcal V_p^p(\mu,\nu) := \inf_{\pi \in \Pi(\mu,\nu)} \int_{\R^d} \left| x - \int_{\R^d} y \, \pi_x(dy) \right|^p \, \mu(dx),
\end{equation}
where we write $(\pi_x)_{x \in \R^d}$ for a disintegration kernel of $\pi$ w.r.t.\ its $\mu$-marginal: $\pi(dx,dy)=\mu(dx)\pi_x(dy)$.
This barycentric weak optimal transport problem has an intrinsic connection with the problem of finding Wasserstein projection.
Indeed, we have that the values of $\mathcal V_p(\mu,\nu)$ and $\W_p(\mu,\mathcal I_p(\mu,\nu))$ coincide, see \cite{AlCoJo20,BaBePa18}.
Moreover, if $\pi^\ast$ is an optimizer of \eqref{eq:def_WOT} then the image of the first marginal $\mu$ under the map $x \mapsto \int_{\R^d} \pi_x^\ast(y)\,dy$ is a minimizer of \eqref{eq:def.Wasserstein_projection.I} and coincides with $\mathcal I_p(\mu,\nu)$ when $p > 1$.
Therefore, when $\mu, \nu \in \mathcal P_p(\R^d)$ are finitely supported, \eqref{eq:def_WOT} can be used to compute the Wasserstein projection.
In particular, ${\mathcal I}_2(\mu,\nu)$ can be computed by solving a quadratic optimization problem with linear constraints.
We refer to \cite{GoRoSaTe17,AlBoCh18} for dual formulations of weak optimal transport problems with additional martingale constraints, to \cite{BaBePa18} for the existence of optimal couplings and necessary and sufficient optimality conditions, to \cite{BaPa19} for continuity of their value function in terms of the marginal distributions $\mu$ and $\nu$, and to \cite{BaPa20} for applications of such problems.
We point out the connection of Wasserstein projections to Cafarelli's contraction theorem that was discovered in \cite{FaGoPr20}.
Note that dual formulations of the minimization problems defining $\mathcal I_p(\mu,\nu)$ and $\mathcal J_p(\mu,\nu)$ have recently been studied by Kim and Ruan \cite{KimRuan}.

\subsection{Wasserstein projection in dimension one}
Wasserstein projections in the special case $d=1$, were further studied in \cite{AlCoJo20} by means of quantile functions.
The quantile function of a probability measure $\mu$ on $\R$ is the left-continuous pseudo-inverse of its cumulative distribution function:
\[
\forall u\in (0,1),\;	F_\mu^{-1}(u):=\inf\{x\in\R:\mu((-\infty,x])\ge u\}.
\]
It is well-known that the comonotonous coupling is an optimizer in \eqref{eq:def.Wasserstein_distance}:
\begin{equation}\textstyle
 \forall \mu,\nu\in\mathcal P_p(\R),\;\W_p(\mu,\nu)=\textstyle{\left(\int_0^1|F_\mu^{-1}(u)-F_\nu^{-1}(u)|^pdu\right)^{1/p}}=:\|F_\mu^{-1}-F_\nu^{-1}\|_p.
  \label{eq:wass1d}
\end{equation} 
Moreover, by \cite[Theorem 3.A.5]{ShSh07} the convex order can be characterised in terms of quantile functions:
for $\mu,\nu \in \mathcal P_1(\R)$ that share the same barycentre we have
\begin{equation}\label{eq:caracc1d}
	\textstyle
   \mu\le_c \nu\iff \forall u\in[0,1],\;\int_0^uF_\mu^{-1}(v)\,dv\ge \int_0^uF_\nu^{-1}(v)\,dv.
\end{equation}
As, according to \cite[Theorem 1.5]{GoRoSaSh18} and \cite[Proposition 4.2]{AlCoJo20}, in one dimension, there exist optimizers of \eqref{eq:def.Wasserstein_projection.I} and \eqref{eq:def.Wasserstein_projection.J} not depending on the power $p \ge 1$, we drop the subscript. Moreover, by \cite[Theorem 2.6 and Proposition 4.2]{AlCoJo20}, the respective quantile functions of such optimizers $\mathcal I(\mu,\nu)$ and $\mathcal J(\mu,\nu)$ are obtained for all $u\in(0,1)$ by
\begin{equation}\label{eq:quantileProjections}
F_{\mathcal I(\mu,\nu)}^{-1}(u)=F_\mu^{-1}(u)-\partial_-\co(G)(u)\text{ and } F_{\mathcal J(\mu,\nu)}^{-1}(u)=F_\nu^{-1}(u)+\partial_-\co(G)(u),
\end{equation}
where $G:v\mapsto\int_0^v(F_\mu^{-1}-F_\nu^{-1})(u)\,du$, $\co$ denotes the convex hull, and $\partial_-$ the left-hand derivative.
The optimizers $\mathcal I_1(\mu,\nu)$ and $\mathcal J_1(\mu,\nu)$ are not necessarily unique, see \cite[Remark 2.3]{AlCoJo20} for $\mathcal I_1$, hence, the maps $\mathcal I$ and $\mathcal J$ provide continuous selections of optimizers for $p=1$. 
A complete geometric characterization of $\mathcal I(\mu,\nu)$ and $\mathcal J(\mu,\nu)$ is given in \cite{BaBePa19}. 
Note that \eqref{eq:quantileProjections} implies
\begin{equation}
   \mathcal W_p(\mathcal I(\mu,\nu),\mu)=\mathcal W_p(\mathcal J(\mu,\nu),\nu)\mbox{ and }\mathcal W_p(\mathcal I(\mu,\nu),\nu)=\mathcal W_p(\mathcal J(\mu,\nu),\mu),\label{eq:equaldist}
 \end{equation}
 where, according to \cite[Corollary 4.4]{AlCoJo20}, the first equality still holds for $d\ge 2$ when $\mathcal I$ and $\mathcal J$ are replaced by $\mathcal I_p$ and $\mathcal J_p$ respectively.
The next examples show that the constants in \eqref{eq:I_Lipschitz} and \eqref{eq:J_Lipschitz} are sharp for $p = 1$.

\begin{example}
	Let $\mu \in \mathcal P_p(\R)$ and $\nu$ be a Dirac measure.
	As $\nu$ is the only measure dominated by itself in the convex order, $\mathcal I(\mu,\nu) = \nu$ and, as a consequence of \eqref{eq:equaldist}, $\mathcal J(\mu,\nu) = \mu$.
	When $\nu'$ is also a dirac mass, we deduce for any $\mu' \in \mathcal P_p(\R)$ that
	\[
		\W_p(\mathcal I(\mu,\nu), \mathcal I(\mu,\nu')) = \W_p(\nu,\nu') \quad \mbox{and} \quad
		\W_p(\mathcal J(\mu,\nu), \mathcal J(\mu',\nu)) = \W_p(\mu,\mu').
	\]
Hence the factor $1$ multiplying $\mathcal W_p(\nu,\nu')$ in the right-hand side of \eqref{eq:I_Lipschitz} and multiplying $\mathcal W_p(\mu,\mu')$ in the right-hand side of \eqref{eq:J_Lipschitz} is optimal.\hfill$\Diamond$
\end{example}
\begin{example}
	We fix $\mu := \delta_0$  and define, for $\alpha \in (0,1)$,
	\[
		\textstyle
		\nu^\alpha := (1-\alpha) \delta_{-\alpha^2} + \alpha \delta_1.
	\]
	We have $\mathcal I(\nu^\alpha,\nu^\alpha) = \nu^\alpha$ and $\mathcal I(\mu,\nu^\alpha) = \delta_{\alpha (1 - \alpha(1 - \alpha))}$, so that
	\begin{align*}
		\textstyle
		\W_1(\mathcal I(\nu^\alpha,\nu^\alpha), \mathcal I(\mu,\nu^\alpha)) &=
		2 \left( \alpha + \alpha^2 (\alpha ( 1 - \alpha) - 1) \right), \\
		\W_1(\mu,\nu^\alpha) &=  \alpha+\alpha^2 (1 - \alpha).	
	\end{align*}
	Then, an application of the de l'H\^opital rule yields
	$\lim_{\alpha \searrow 0}
		\frac{\W_1(\mathcal I(\nu^\alpha,\nu^\alpha), \mathcal I(\mu,\nu^\alpha))}{\W_1(\mu,\nu^\alpha)} = 2$.
	Hence, the factor 2 in \eqref{eq:I_Lipschitz} is optimal when $p = 1$.
	Since $\mathcal J(\delta_{\alpha ( 1 - \alpha(1- \alpha))},\nu^\alpha)=\nu^\alpha$ and $\mathcal J(\delta_{\alpha(1-\alpha(1-\alpha))},\delta_0)=\delta_{\alpha(1-\alpha(1 - \alpha))}$, we find in the same way that the factor 2 is also optimal in \eqref{eq:J_Lipschitz} when $p = 1$.\hfill$\Diamond$
\end{example}
\begin{example}
	Let $\mu,\mu',\nu,\nu'$ be the probability measures with quantile functions: 
	\begin{gather*}
		\textstyle
		F_\mu^{-1}(u)=u\mathbbm 1_{(0,\frac 12]}(u)+\frac{1+u}{2}\mathbbm 1_{(\frac 12,1)}(u), \quad 
		F_\nu^{-1}(u)=\frac u2, \\
		\textstyle
		F_{\mu'}^{-1}(u)=u\mathbbm 1_{(0,\frac 12]}(u)+\frac{12+5u}{18}\mathbbm 1_{(\frac 12,1)}(u),\quad
		F_{\nu'}^{-1}(u)=\frac u3\mathbbm 1_{(0,\frac 12]}(u)+\frac u2\mathbbm 1_{(\frac 12,1)}(u).	
	\end{gather*}
	We check that $F_{\mathcal I(\mu,\nu)}^{-1}(u)=\frac u2$ and $F_{\mathcal I(\mu',\nu')}^{-1}(u)=\frac u3 \mathbbm 1_{(0,\frac12]}+\frac{3+5u}{18}\mathbbm 1_{(\frac 12,1)}(u)$, whence,
	\[ 
		\mathcal W^p_p(\mathcal I(\mu,\nu),\mathcal I(\mu',\nu'))=\mathcal W^p_p(\mu,\mu')+\mathcal W^p_p(\nu,\nu')
	\]
	with two positive summands.\hfill $\Diamond$
\end{example}

\subsection{On the convex-order lattice in dimension one}
 The restauration of the convex order which was the original motivation for the introduction of the Wasserstein  projections can also be achieved in dimension one using that $\mathcal P_1(\R)$ is a complete lattice for the increasing and decreasing convex orders (see \cite{KeRo00}). They both coincide with the convex order on the subset $\mathcal P^{x_0}_p(\R)$ of $\mathcal P_p(\R)$ consisting in probability measures with barycentre $x_0 \in \R$. 
On $\mathcal P^{x_0}_1(\R)$, the minimum $\wedge_c$ and maximum $\vee_c$ can be expressed in terms of potential functions: 
the potential function of $\mu \in \mathcal P_1(\R)$ is defined by
\[
	\textstyle
	u_\mu(x) := \int_\R |x - y| \, \mu(dy).
\]
For $\mu,\nu \in \mathcal P^{x_0}_1(\R)$, $\mu\wedge_c\nu$ and $\mu\vee_c\nu$ are uniquely determined by
\[
	u_{\mu \wedge_c \nu} = \co(u_\mu \wedge u_\nu) \quad\mbox{and}\quad
	u_{\mu \vee_c \nu} = u_\mu \vee u_\nu.
\]
On the domain $\mathcal P^{x_0}_p(\R) \times \mathcal P^{x_0}_p(\R)$,
\[
	(\mu,\nu) \mapsto \mu \wedge_c \nu \quad \mbox{and} \quad (\mu,\nu) \mapsto \mu\vee_c \nu,
\]
are continuous mappings into $\mathcal P^{x_0}_p(\R)$:
\cite[Lemma 4.1]{BeJoMaPa21a} provides continuity for $p = 1$ and \cite[Lemma 4.3]{AlCoJo20}, which ensures uniform integrability, permits to deduce continuity for general $p \ge 1$.
However, unlike $\mathcal I$ and $\mathcal J$, the minimum and maximum in the convex order are not Lipschitz continuous.

\begin{example}
	Consider for $n\ge3$ the measures in $\mathcal P_p(\R)$:

\begin{minipage}{0.45\textwidth}
	\begin{align*}\textstyle
		\nu:=& \frac1{2n}\delta_0 + \frac1n \sum_{i = 1}^{n - 1} \delta_{\frac{i}n} + \frac1{2n} \delta_1, \\
		\textstyle
		\mu :=& \frac1{n} \sum_{i = 1}^n \delta_{\frac{2i - 1}{2n}}, \\
		\textstyle
		\eta :=& \frac 3{2n}\delta_{\frac1{n}} + \frac1n \sum_{i = 2}^{n - 2} \delta_{\frac{i}n} + \frac 3{2n} \delta_{\frac{n - 1}{n}}.
	\end{align*}
\end{minipage}\hfill
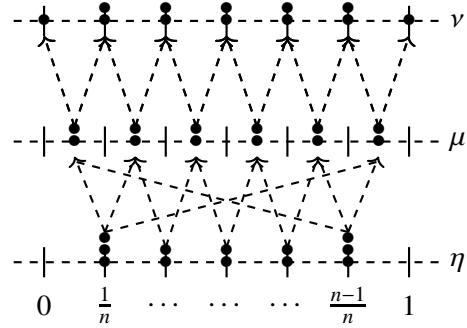
\begin{minipage}{0.45\textwidth}
	\centering
	\begin{tikzpicture}[scale=0.8]
	\draw[thick,dashed] (-0.5,0)--(6.5,0) node[right] {$\nu$};
	\draw[thick] (0,-0.25) -- (0,0.25);
	\draw[thick] (1,-0.25) -- (1,0.25);
	\draw[thick] (2,-0.25) -- (2,0.25);
	\draw[thick] (3,-0.25) -- (3,0.25);
	\draw[thick] (4,-0.25) -- (4,0.25);
	\draw[thick] (5,-0.25) -- (5,0.25);
	\draw[thick] (6,-0.25) -- (6,0.25);	
	\node at (0,0) {$\bullet$};
	\node at (1,0) {$\bullet$};
	\node at (2,0) {$\bullet$};
	\node at (3,0) {$\bullet$};
	\node at (4,0) {$\bullet$};
	\node at (5,0) {$\bullet$};
	\node at (6,0) {$\bullet$};
	\node at (1,0.2) {$\bullet$};
	\node at (2,0.2) {$\bullet$};
	\node at (3,0.2) {$\bullet$};
	\node at (4,0.2) {$\bullet$};
	\node at (5,0.2) {$\bullet$};
	
	\draw[thick,dashed,->] (0.5,-1.75) -- (0,-0.25);
	\draw[thick,dashed,->] (0.5,-1.75) -- (1,-0.25);
	\draw[thick,dashed,->] (1.5,-1.75) -- (1,-0.25);
	\draw[thick,dashed,->] (1.5,-1.75) -- (2,-0.25);
	\draw[thick,dashed,->] (2.5,-1.75) -- (2,-0.25);
	\draw[thick,dashed,->] (2.5,-1.75) -- (3,-0.25);
	\draw[thick,dashed,->] (3.5,-1.75) -- (3,-0.25);
	\draw[thick,dashed,->] (3.5,-1.75) -- (4,-0.25);
	\draw[thick,dashed,->] (4.5,-1.75) -- (4,-0.25);
	\draw[thick,dashed,->] (4.5,-1.75) -- (5,-0.25);
	\draw[thick,dashed,->] (5.5,-1.75) -- (5,-0.25);
	\draw[thick,dashed,->] (5.5,-1.75) -- (6,-0.25);
	
	\draw[thick,dashed,->] (0.5,-1.75) -- (0,-0.25);
	\draw[thick,dashed,->] (0.5,-1.75) -- (1,-0.25);
	\draw[thick,dashed,->] (1.5,-1.75) -- (1,-0.25);
	\draw[thick,dashed,->] (1.5,-1.75) -- (2,-0.25);
	\draw[thick,dashed,->] (2.5,-1.75) -- (2,-0.25);
	\draw[thick,dashed,->] (2.5,-1.75) -- (3,-0.25);
	\draw[thick,dashed,->] (3.5,-1.75) -- (3,-0.25);
	\draw[thick,dashed,->] (3.5,-1.75) -- (4,-0.25);
	\draw[thick,dashed,->] (4.5,-1.75) -- (4,-0.25);
	\draw[thick,dashed,->] (4.5,-1.75) -- (5,-0.25);
	\draw[thick,dashed,->] (5.5,-1.75) -- (5,-0.25);
	\draw[thick,dashed,->] (5.5,-1.75) -- (6,-0.25);

	\begin{scope}[shift={(0,-2)}]
	\draw[thick,dashed] (-0.5,0)--(6.5,0) node[right] {$\mu$};
	\draw[thick] (0,-0.25) -- (0,0.25);
	\draw[thick] (1,-0.25) -- (1,0.25);
	\draw[thick] (2,-0.25) -- (2,0.25);
	\draw[thick] (3,-0.25) -- (3,0.25);
	\draw[thick] (4,-0.25) -- (4,0.25);
	\draw[thick] (5,-0.25) -- (5,0.25);
	\draw[thick] (6,-0.25) -- (6,0.25);
	\node at (0.5,0) {$\bullet$};
	\node at (1.5,0) {$\bullet$};
	\node at (2.5,0) {$\bullet$};
	\node at (3.5,0) {$\bullet$};
	\node at (4.5,0) {$\bullet$};
	\node at (5.5,0) {$\bullet$};
	\node at (0.5,0.2) {$\bullet$};
	\node at (1.5,0.2) {$\bullet$};
	\node at (2.5,0.2) {$\bullet$};
	\node at (3.5,0.2) {$\bullet$};
	\node at (4.5,0.2) {$\bullet$};
	\node at (5.5,0.2) {$\bullet$};
	
	\draw[thick,dashed,->] (1,-1.5) -- (0.5,-0.25);
	\draw[thick,dashed,->] (1,-1.5) -- (1.5,-0.25);
	\draw[thick,dashed] (1,-1.5) -- (5.5,-0.33);
	\draw[thick,dashed,->] (2,-1.75) -- (1.5,-0.25);
	\draw[thick,dashed,->] (2,-1.75) -- (2.5,-0.25);
	\draw[thick,dashed,->] (3,-1.75) -- (2.5,-0.25);
	\draw[thick,dashed,->] (3,-1.75) -- (3.5,-0.25);
	\draw[thick,dashed,->] (4,-1.75) -- (3.5,-0.25);
	\draw[thick,dashed,->] (4,-1.75) -- (4.5,-0.25);
	\draw[thick,dashed,->] (5,-1.5) -- (4.5,-0.25);
	\draw[thick,dashed,->] (5,-1.5) -- (5.5,-0.25);
	\draw[thick,dashed] (5,-1.5) -- (0.5,-0.33);
	\end{scope}
	
	\begin{scope}[shift={(0,-4)}]
	\draw[thick,dashed] (-0.5,0)--(6.5,0) node[right] {$\eta$};
	\draw[thick] (0,-0.25) -- (0,0.25) node[yshift=-2em] {$0$};
	\draw[thick] (1,-0.25) -- (1,0.25) node[yshift=-2em] {$\frac1n$};
	\draw[thick] (2,-0.25) -- (2,0.25) node[yshift=-2em] {$\dots$};
	\draw[thick] (3,-0.25) -- (3,0.25) node[yshift=-2em] {$\dots$};
	\draw[thick] (4,-0.25) -- (4,0.25) node[yshift=-2em] {$\dots$};
	\draw[thick] (5,-0.25) -- (5,0.25) node[yshift=-2em] {$\frac{n-1}n$};
	\draw[thick] (6,-0.25) -- (6,0.25) node[yshift=-2em] {$1$};
	\node at (1,0.4) {$\bullet$};
	\node at (1,0) {$\bullet$};
	\node at (2,0) {$\bullet$};
	\node at (3,0) {$\bullet$};
	\node at (4,0) {$\bullet$};
	\node at (5,0) {$\bullet$};
	\node at (5,0.4) {$\bullet$};
	\node at (1,0.2) {$\bullet$};
	\node at (2,0.2) {$\bullet$};
	\node at (3,0.2) {$\bullet$};
	\node at (4,0.2) {$\bullet$};
	\node at (5,0.2) {$\bullet$};
	\end{scope}
	\end{tikzpicture}
\end{minipage}%
\vspace{1em}
Observe that, for the martingale kernel $K(x,dy) := \frac12 \delta_{x - \frac1{2n}} + \frac12 \delta_{x + \frac1{2n}}$, $\mu(dx)\, K(x,dy) \in \Pi(\mu,\nu)$. By Strassen's theorem \cite{St65} we find that $\mu \le_c \nu$, $\eta \le_c \tilde \mu$ where $\tilde \mu(dy):=\int_{x\in\R}K(x,dy)\,\eta(dx)$ is such that $\mu-\tilde\mu=\frac 1{4n}\delta_{\frac 1{2n}}+\frac 1{4n}\delta_{\frac{2n-1}{2n}}-\frac 1{4n}\delta_{\frac 3{2n}}-\frac 1{4n}\delta_{\frac{2n-3}{2n}}$ so that $\tilde \mu \le_c \mu$. Hence $\eta \le_c \mu \le_c \nu$ and $\mu \vee_c \nu=\nu$, $\mu\vee_c \eta=\mu$, $\mu\wedge_c \nu=\mu$, $\mu\wedge_c \eta=\eta$.
We compute
\[	\textstyle
	\W_p(\mu,\nu) = \W_p(\eta,\mu) = \frac{1}{2n} \quad \mbox{and} \quad
	\W_p(\eta,\nu) = \frac{1}{n^{1+1/p}},
\]
from where we conclude that
\begin{equation}
	\textstyle
	\frac{\W_p(\mu \vee_c \nu, \mu\vee_c \eta)}{\W_p(\eta,\nu)} = 
	\frac{\W_p(\mu\wedge_c \nu, \mu\wedge_c \eta)}{\W_p(\eta,\nu)} =
	\frac{n^{1 / p}}{2}.
\end{equation}
Consequentially, $\wedge_c$ and $\vee_c$ are not Lipschitz continuous.\hfill$\Diamond$
\end{example}

In particular, this example shows that for probability measures with the same barycenter, in general, the Wasserstein projections do not coincide with the minimum and maximum in the convex order.
Beyond that, they satisfy the following order relation.

\begin{proposition}\label{prop:IJordering}
	Let $\mu,\nu \in \mathcal P_p(\R)$ have the same barycenter. Then
	\begin{equation}
		\mathcal I(\mu,\nu) \le_c \mu \wedge_c \nu
		\quad\mbox{and}\quad
		\mathcal J(\mu,\nu) \ge_c \mu \vee_c \nu.
	\end{equation}
\end{proposition}

\begin{proof}[Proof of Proposition \ref{prop:IJordering}]
	The map $u\mapsto F_{\mathcal I(\mu,\nu)}^{-1}(u)-F_\mu^{-1}(u)$ is non-increasing due to \eqref{eq:quantileProjections}.
	Therefore, monotonicity of the integrand yields, for $u \in (0,1)$ that
	\[ \textstyle
	\frac1 u\int_0^u(F_{\mathcal I(\mu,\nu)}^{-1}(v)-F_\mu^{-1}(v))\,dv\ge\int_0^1(F_{\mathcal I(\mu,\nu)}^{-1}(v)-F_\mu^{-1}(v))\,dv=\int_\R y\,\nu(dy)-\int_\R x\,\mu(dx),
	\]
	where the last equality comes from the inverse transform sampling and the fact that, as $\mathcal I(\mu,\nu) \le_c \nu$, $\mathcal I(\mu,\nu)$ and $\nu$ share the same barycentre.
	If $\mu$ and $\nu$ have the same barycenter, or put equivalently, $\int_\R y\,\nu(dy)=\int_\R x\,\mu(dx)$, then $\mathcal I(\mu,\nu)$ shares this barycenter and we deduce by \eqref{eq:caracc1d} that $\mathcal I(\mu,\nu)\le_c\mu$, thus, $\mathcal I(\mu,\nu)\le_c\mu\wedge_c\nu$. 
	Analogously, we have $\mu\le_c\mathcal J(\mu,\nu)$, and if $\mu$ and $\nu$ share the same barycentre,  $\nu\le_c\mathcal J(\mu,\nu)$, hence, $\mu\vee_c\nu\le\mathcal J(\mu,\nu)$.
\end{proof}

\section{Proof of Theorem \ref{thm:main_result}}
The proof of Theorem \ref{thm:main_result} relies on the next two results whose proofs are postponed.
\begin{lemma}
	\label{lem:continuity_W_projection}
	For $p\ge 1$, $\mathcal I$ and $\mathcal J$ are continuous maps on $\mathcal P_p(\R) \times \mathcal P_p(\R)$ to $\mathcal P_p(\R)$.
\end{lemma}
\begin{proposition} 
	\label{prop:convex_hull_contracting}
	Let $f$ and $g$ be real-valued càdlàg functions on $[0,1]$ with respective antiderivatives $F$ and $G$.
	We have, for $p \ge 1$,
	\begin{equation}
		\label{eq:convex_hull_contracting.assertion}
		\left\Vert \partial_+ \left(\co(F) - \co(G)\right) \right\Vert_p \leq \left\Vert f - g\right\Vert_p.
	\end{equation}	
\end{proposition}
\begin{proof}[Proof of Theorem \ref{thm:main_result}]
	Let $\mu,\mu',\nu,\nu' \in \mathcal P_p(\R)$.
	Assume for a moment that \eqref{eq:I_Lipschitz} and \eqref{eq:J_Lipschitz} hold for probability measures with bounded support.
	Since such measures are dense in $\mathcal P_p(\R)$, there exist $\mu_n,\mu'_n,\nu_n,\nu'_n \in \mathcal P_p(\R)$, $n\in\N$ with bounded support such that
	\begin{equation}\label{convergenceWrmunnun}
	\lim_{n\to+\infty} \mathcal W_p(\mu,\mu_n) + \mathcal W_p(\mu',\mu'_n) + \mathcal W_p(\nu,\nu_n) + \mathcal W_p(\nu',\nu'_n) = 0.
	\end{equation}
	We have by Lemma \ref{lem:continuity_W_projection} that $(\mathcal J(\mu_n,\nu_n))_{n \in \N}$ and $(\mathcal J(\mu'_n,\nu'_n))_{n \in \N}$ converge to $\mathcal J(\mu,\nu)$ and $\mathcal J(\mu',\nu')$ resp.\ in $\mathcal P_p(\R)$.
	Therefore,
	\begin{align*}
	\mathcal W_p\left(\mathcal J(\mu,\nu), \mathcal J(\mu',\nu')\right)
	&= 
	\lim_{n\to+\infty} \mathcal W_p\left(\mathcal J(\mu_n,\nu_n),\mathcal J(\mu'_n,\nu'_n)\right) 
	\\
	&\leq 
	\lim_{n\to+\infty} 2\mathcal W_p(\mu_n,\mu'_n) + \mathcal W_p(\nu_n,\nu'_n) 
	\\
	&= 2 \mathcal W_p(\mu,\mu') + \mathcal W_p(\nu,\nu').
	\end{align*}	
	Hence, we may assume that $\mu,\nu,\mu',\nu'$ have bounded supports. 
	This implies that the associated quantile functions are bounded on $(0,1)$ and since they are non-decreasing and have at most countably many jumps, coincide $\lambda$-a.s.\ with càdlàg functions on $[0,1]$, where $\lambda$ denotes the Lebesgue measure on $[0,1]$.
	Therefore,
	\[ \textstyle
		G:v\mapsto\int_0^v( F_{\mu}^{-1}- F_{\nu}^{-1})(u) \, du\quad\text{and}\quad G':v\mapsto\int_0^v (F_{\mu'}^{-1} - F_{\nu'}^{-1})(u)\,du,
    \]
	are the  antiderivatives of c\`adl\`ag functions on $[0,1]$. By Proposition \ref{prop:convex_hull_contracting} and \eqref{eq:wass1d},
	\begin{align*}
		\Vert  \partial_+\left(\co(G) - \co(G')\right)\Vert_p&\le\Vert \partial_+(G-G') \Vert_p=\Vert F_\mu^{-1} - F_\nu^{-1} - F_{\mu'}^{-1} + F_{\nu'}^{-1} \Vert_p
		\\ &\le \Vert F_\mu^{-1} - F_{\mu'}^{-1} \Vert_p + \Vert F_\nu^{-1} - F_{\nu'}^{-1} \Vert_p = \mathcal W_p(\mu,\mu') + \mathcal W_p(\nu,\nu').		
	\end{align*}
   By \eqref{eq:quantileProjections}, we have for $\lambda$-almost every $u \in (0,1)$ that
	\begin{align*}
	&F_{\mathcal I(\mu,\nu)}^{-1}(u) = F_\mu^{-1}(u)-\partial_+\co(G)(u)\mbox{ and }F_{\mathcal I(\mu',\nu')}^{-1}(u) = F_{\mu'}^{-1}(u)-\partial_+\co(G')(u),\\
		&F_{\mathcal J(\mu,\nu)}^{-1}(u) = F_\nu^{-1}(u)+\partial_+\co(G)(u)\mbox{ and }	F_{\mathcal J(\mu',\nu')}^{-1}(u) = F_{\nu'}^{-1}(u)+\partial_+\co(G')(u).
	\end{align*}
Therefore, using \eqref{eq:wass1d}, we obtain
	\begin{align*}
	\text{$\textstyle \mathcal W_p(\mathcal I(\mu,\nu),\mathcal I(\mu',\nu')) $}&\text{$\textstyle =\left\Vert  F_{\mathcal I(\mu,\nu)}^{-1}-F_{\mathcal I(\mu',\nu')}^{-1}\right\Vert_p = \left\Vert  F_\mu^{-1} -\partial_+ \co(G) - F_{\mu'}^{-1} +\partial_+ \co(G')\right\Vert_p$} \\
	&\text{$\textstyle \leq \left\Vert F_\mu^{-1} - F_{\mu'}^{-1} \right\Vert_p + \left\Vert  \partial_+\left(\co(G) - \co(G')\right)\right\Vert_p \leq 2\mathcal W_p(\mu,\mu') + \mathcal W_p(\nu,\nu').$} 
	\end{align*}
In the same way, \begin{align*}
   \mathcal W_p(\mathcal J(\mu,\nu),\mathcal J(\mu',\nu')) &= \left\Vert  F_\nu^{-1} + \partial_+ \co(G) - F_{\nu'}^{-1} - \partial_+ \co(G')\right\Vert_p \\&\leq \mathcal W_p(\mu,\mu') + 2\mathcal W_p(\nu,\nu').\qedhere
\end{align*}
\end{proof}
\begin{proof}[Proof of Lemma \ref{lem:continuity_W_projection}]
	In the following we will only show continuity of $\mathcal J$ and remark that continuity of $\mathcal I$ follows mutatis mutandis (and can be even shown with a simpler line of argument, since $\mathcal I(\mu,\nu) \le_{c} \nu$).
	Let $(\mu_n)_{n \in \N}$, $(\nu_n)_{n \in \N}$ be sequences that converge to $\mu_\infty$, $\nu_\infty$ resp.\ in $\mathcal P_p(\R)$.

	\emph{Step 1}. We show that $(\mathcal J(\mu_n,\nu_n))_{n \in \N}$ is a precompact subset of $\mathcal P_p(\R)$.
	As a consequence of the de la Vallée-Poussin theorem, see for example \cite[Theorem 4.5.9 and proof]{Bo07}, there exists a continuous, increasing and strictly convex function $\theta \colon \R_+ \to \R_+$ such that $\lim_{r \to \infty} \frac{r}{\theta(r)} = 0$ and
	\[ \textstyle
		\sup_{n \in \N} \int \theta(|x|^p) \,	\mu_n(dx) < \infty, \quad
		\sup_{n \in \N} \int \theta(|y|^p) \, \nu_n(dy) < \infty.
	\]
	Note that, when $p>1$, $\theta \circ (| \cdot |^p / 2^p)$ is also strictly convex as the composition of a convex, increasing function with a strictly convex function. 
	On the other hand, when $p = 1$, $x\neq y$ and $\alpha\in (0,1)$, either $|x|\neq |y|$ or the inequality $|\alpha x+(1-\alpha)y|\le \alpha|x|+(1-\alpha)|y|$ is strict, so that, since $\theta$ is increasing and strictly convex, $\theta  \circ (| \cdot | / 2)$ is again strictly convex.
	We conclude that in any case $\theta \circ (| \cdot |^p / 2^p)$ is strictly convex.

	Consider the transport problem $\mathcal W_\theta$ given by
	\begin{equation}\textstyle
		\label{eq:continuity_W_projection.W_theta}
		\mathcal W_\theta(\eta,\nu) := \inf_{ \pi \in \Pi(\mu, \nu)} \int \theta \left( \frac{|x - y|^p}{2^p} \right) \, \pi(dx,dy),
	\end{equation}
	for $\eta,\nu \in \mathcal P_p(\R)$, and observe that
	\[ \textstyle
		\mathcal W_\theta(\eta,\nu) 
		\le \int \theta \left( \frac{|x|^p + |y|^p}{2} \right) \, \eta \otimes \nu(dx,dy) 
		\le \frac12 \left( \int \theta(|x|^p) \, \eta(dx) + \int \theta(|y|^p) \, \nu(dy) \right).
	\]
	We have by \cite[Theorem 1.4]{BaBePa19} that
	\begin{align*}
		\mathcal V_\theta(\mu,\nu) := \inf_{\mu \le_{c} \eta} \mathcal W_\theta(\eta,\nu) = \mathcal W_\theta(\mathcal J(\mu, \nu), \nu) \le \mathcal W_\theta(\mu, \nu),
	\end{align*}
	from which we deduce that $(\mathcal V_\theta(\mu_n,\nu_n))_{n \in \N}$ is a bounded sequence.
	For $n \in \N$, let $\pi_n$ be an optimizer of $\mathcal W_\theta(\mathcal J(\mu_n,\nu_n), \nu_n)$ in $\Pi(\mathcal J(\mu_n,\nu_n), \nu_n)$. We then find by monotonicity and convexity of $\theta$
	\begin{align*}
		\text{$\textstyle \int \theta \left( \frac{|x|^p}{4^p} \right) \, \mathcal J(\mu_n,\nu_n)(dx)$}&\text{$\textstyle \le \int \theta \left( \frac{|x-y|^p+|y|^p}{2\times 2^p} \right) \, \pi_n(dx,dy) $}\\
		&\text{$\textstyle \le
		\frac12 \left( \mathcal V_\theta(\mu_n,\nu_n) + \int \theta\left(\frac{|y|^p}{2^p}\right)\nu_n(dy)\right)$}
		\\
		&\text{$\textstyle \le\frac12 \left( \mathcal V_\theta(\mu_n,\nu_n) + \int \theta\left(|y|^p\right)\nu_n(dy)\right).$}
	\end{align*}
	Therefore, the left-hand side is uniformly bounded in $n \in \N$.
	Recall that $\frac{r}{\theta(r)}$ vanishes for $r \to \infty$ and so does $\sup_{s \ge r} \frac{s}{\theta(s)}$.
	% By convexity of $\theta$, for $r>0$ large enough so that $\theta(r)>0$, one has
    %     $$\forall s\ge r,\;s\le \frac{r}{\theta(r)}\theta(s)+\frac{s-r}{\theta(r)}\theta(0)\le \frac{r}{\theta(r)}\theta(s)+\frac{\theta(0)}{\theta(r)}s.$$
	Since % As $\theta$ is non-negative we derive
	\[ \textstyle
		\int \mathbbm 1_{[r,\infty)} \left(\frac{|x|^p}{4^p}\right) \frac{|x|^p}{4^p} \, \mathcal J(\mu_n,\nu_n) (dx)
		\le \sup_{s \ge r}\frac{s}{\theta(s)}\int \theta\left(\frac{|x|^p}{4^p}\right)\, \mathcal J(\mu_n,\nu_n)(dx),
	\]
	% \begin{align*}
	% 	\int \mathbbm 1_{[r,\infty)} \left(\frac{|x|^p}{4^p}\right) \frac{|x|^p}{4^p} \, \mathcal J(\mu_n,\nu_n)(dx)&\le \frac{r}{\theta(r)}\int \theta\left(\frac{|x|^p}{4^p}\right)\, \mathcal J(\mu_n,\nu_n)(dx)\\&+\frac{\theta(0)}{\theta(r)}\int \frac{|x|^p}{4^p} \, \mathcal J(\mu_n,\nu_n)(dx).
	% \end{align*}
with the integrals on the right-hand side uniformly bounded in $n \in \N$, we deduce that
	\[ \textstyle
		\lim_{r \to \infty}\sup_{n \in \N} \int \mathbbm 1_{[r,\infty)}(|x|^p) |x|^p \, \mathcal J(\mu_n,\nu_n)(dx) 	% \le \lim_{r \to \infty} \frac{r}{\theta(r)} \sup_{n \in \N} \int \theta(|x|^p) \, \mathcal J(\mu_n,\nu_n)(dx) 
		= 0.
	\]
	In particular, by Markov's inequality, we get that the sequence is tight and by \cite[Definition 6.8]{Vi09} precompact in $\mathcal P_p(\R)$.

	\emph{Step 2}. Precompactness allows us to pass to a subsequence convergent in $\mathcal P_p(\R)$ with limit $\gamma$.
	Consider the continuous, increasing, and strictly convex function $\hat \theta(x) := \sqrt{x^{2p} + 1}$ on $\R_+$ with $\hat \theta(x) \le x^p + 1$.
	By stability, that is \cite[Theorem 1.5]{BaBePa19}, we obtain
	\[
		\mathcal V_{\hat \theta}(\mu,\nu) = \lim_{n \to \infty} \mathcal V_{\hat \theta}(\mu_n,\nu_n)
		= \lim_{n \to \infty} \mathcal W_{\hat \theta}(\mathcal J(\mu_n,\nu_n),\nu_n) = \mathcal W_{\hat \theta}(\gamma, \nu).	
	\]
	Since $\mathcal W_p$-convergence preserves convex ordering, we get that $\mu \le_{c} \gamma$ and by uniqueness of the optimizer of $\mathcal V_{\hat \theta}(\mu,\nu)$, $\gamma = \mathcal J(\mu,\nu)$.
	Hence, $(\mathcal J(\mu_n,\nu_n))_{n \in \N}$ converges to $\mathcal J(\mu,\nu)$ in $\mathcal P_p(\R)$.
      \end{proof}
      The proof of Proposition \ref{prop:convex_hull_contracting} relies on the next three lemmas. The proof of the first one is postponed after the one of the proposition.
\begin{lemma} \label{lem:convex_affine}
	Let $0\le a<b$ and $F,G:[0,b]\to\R$ be continuous on $[0,b]$, convex on $[0,a)$ and affine on $[a,b]$. Then for any increasing convex function $\theta \colon \R \to \R$  we have
	\begin{equation}\label{eq:convex_affine.assertion}
	\int_0^b \theta\left( \vert \partial_+(\co(F)-\co(G))\vert \right)(u) \, du \leq \int_0^b \theta\left( \vert \partial_+(F-G) \vert \right)(u) \, du
	\end{equation}
	where $\co$ denotes the convex hull and $\partial_+$ the right-hand derivative.
\end{lemma}
\begin{lemma}\label{lem:convex hull composition} Let $0\le a<b<\infty$, $F:[0,b]\to\R$, and
	\[ 
		H(x):=\begin{cases}
   \co(F\vert_{[0,a)})(x)\mbox{ if }x\in[0,a)\\F(x)\mbox{ if } x\in[a,b].
\end{cases}
\]
Then $\co(H)=\co(F)$.
\end{lemma}
\begin{proof}[Proof of Lemma \ref{lem:convex hull composition}]
The function $\co(H)$ is convex and satisfies $\co(H)\le H\le F$. By definition of the convex hull, we deduce that $\co(H)\le \co(F)$. Conversely, $\co(F)$ is convex and bounded from above by $F$, so that the restriction $\co(F)\vert_{[0,a)}$ is convex and bounded from above by $F\vert_{[0,a)}$. Hence $\co(F)\vert_{[0,a)}\le \co(F\vert_{[0,a)})$ and $\co(F)\le H$. By definition of the convex hull, $\co(F)\le\co(H)$, which concludes the proof.
\end{proof}
\begin{lemma}
	\label{lem:cadlag_approx}
	Let $f \colon [0,1] \to \R$ be a c\`adl\`ag function.
	Then there exists for every $\epsilon > 0$ a piecewise constant, c\`adl\`ag function $g \colon [0,1] \to \R$ with at most finitely many jumps such that $\sup_{x \in [0,1]} |f(x) - g(x)| < \epsilon$.
\end{lemma}

\begin{proof}[Proof of Lemma \ref{lem:cadlag_approx}]
	This is follows from \cite[Section 12, Lemma 1]{Bi99} and the discussion below.
% 	Fix $\epsilon > 0$, set $x_0 := 0$, and define recursively, for $k \in \N$ and as long as $\{ x \in [x_{k - 1}, 1] \colon |f(x) - f(x_{k - 1})| \ge \epsilon \} \neq \emptyset$,
% 	\[	
% 		x_k := \inf \{ x \in [x_{k - 1}, 1] \colon |f(x) - f(x_{k - 1})| \ge \epsilon \},
% 	\]
% 	which is attained and greater than $x_{k-1}$ by right-continuity of $f$.
% 	We claim that this algorithm terminates after finitely many ($N \in \N$) steps.
% 	Assume the opposite, then there is a sequence $(x_k)_{k \in \N}$ with $x_k \nearrow x^\ast \le 1$ and $\epsilon \le \liminf_{k \to \infty} |f(x_k) - f(x_{k + 1})|$ which yields a contradiction to the existence of the left limit of $f$ at $x^\ast$.
% 	Hence, there exists $N\in\N$ such that $(x_k)_{0\le k\le N}$ is an increasing sequence and $\forall x\in[x_N,1]$, $|f(x)-f(x_N)|<\epsilon$. The function $g$ defined by
% 	\[
% 		g(x) := \sum_{k = 0}^{N} \mathbbm 1_{[x_k, x_{k + 1})}(x) f(x_{k}) + \mathbbm 1_{\{ 1\}}(x) f(x_N)\mbox{ where }x_{N + 1} := 1,
% 	\]
% has the desired properties.
\end{proof}

\begin{proof}[Proof of Proposition \ref{prop:convex_hull_contracting}]
	For the moment we assume that the assertion of the proposition holds true for antiderivatives of piecewise constant, càdlàg functions.
	Since $f$ and $g$ are càdlàg, there exist by Lemma \ref{lem:cadlag_approx} for every $n \in \N^*$ piecewise constant, càdlàg functions $f_n, g_n \colon [0,1]\to\R$ with finitely many discontinuities such that
	\[
		\sup_{x \in [0,1] }
		\left\{\vert f(x)-f_n(x)\vert+\vert g(x)- g_n(x)\vert\right\}<1/n.
	\]
% \textcolor{blue}{Reference? For $x\in[0,1]$, we choose $\varepsilon_x>0$ such that \begin{align*}
%    &\forall y\in[0,1]\cap (x-\varepsilon_x,x),\;\left|f(y)-f(x-)\right|<1/n\mbox{ and }\\&\forall y\in[0,1]\cap [x,x+\varepsilon_x),\;\left|f(y)-f(x)\right|<1/n.\end{align*} From the open cover $((x-\varepsilon_{x},x+\varepsilon_{x}))_{x\in[0,1]}$ of the compact set $[0,1]$, we extract a finite cover $((x_k-\varepsilon_{x_k},x_k+\varepsilon_{x_k}))_{1\le k\le K}$, where we may suppose that $x_1<x_2<\cdots<x_K$. Then $[0,1]=\bigcup_{k=1}^K\left\{[0,1]\cap(x_k-\varepsilon_{x_k},x_k+\varepsilon_{x_k})\right\}$ where, up to removing some sets in the union, we may suppose that no set of the union is included in another one and $[0,1]\neq\bigcup_{k=1}^{K-1}\left\{[0,1]\cap(x_k-\varepsilon_{x_k},x_k+\varepsilon_{x_k})\right\}$. Then, for $1\le k\le K-1$, $\bigcup_{j=1}^{k}\left\{[0,1]\cap(x_j-\varepsilon_{x_j},x_j+\varepsilon_{x_j})\right\}=[0,x_k+\varepsilon_{x_k})$ with $x_k+\varepsilon_{x_k}<1$. With the convention $x_0+\varepsilon_{x_0}=0$, for $x\in[0,1]$, we set
%  \begin{align*}
%    f_n(x)=\sum_{k=1}^K\left(\mathbbm 1_{(x_k-\varepsilon_{x_k},x_k)\cap[x_{k-1}+\varepsilon_{x_{k-1}},1]}(x)f(x_k-)+\mathbbm 1_{[x_k,x_k+\varepsilon_{x_k})\cap[x_{k-1}+\varepsilon_{x_{k-1}},1]}(x)f(x_k)\right).
%  \end{align*} For $1\le k\le K$, we choose $f_n$ equal to $f(x_k-)$ on $[0,1]\cap(x_k-\varepsilon_{x_k},x_k)\cap[x_{k-1}+\varepsilon_{x_k},1]$ and to $f(x_k)$ on $[0,1]\cap[x_k,x_k+\varepsilon_{x_k})\cap[x_{k-1}+\varepsilon_{x_k},1]$.}	
	Therefore, $(f_n)_{n \in \N^*}$ and $(g_n)_{n \in \N^*}$ converge in $L^p(\lambda)$ to $f$ and $g$, respectively.
	Let, for $n \in \N^*, u \in [0,1]$, 
	\[
		F_n(u) := F(0) + \int_0^u f_n(v) \, dv \mbox{ and }G_n(u) := G(0) + \int_0^u g_n(v) \, dv.
	\]
	We have $\| F - F_n \|_\infty <\frac1n$, $\| G - G_n \|_\infty < \frac1n$. Since $\co(F)-\| F - F_n \|_\infty$ (resp.\ $\co(F_n)-\| F - F_n \|_\infty$) is a convex function bounded from above by $F-\| F - F_n \|_\infty\le F_n$ (resp.\ $F$), $\| \co(F) - \co(F_n) \|_\infty \le \| F - F_n \|_\infty < \frac1n$ and, in the same way, $\| \co(G) - \co(G_n) \|_\infty < \frac1n$.
	By \cite[Theorem 6.2.7]{HiLe04}, we have $\lambda$-almost sure convergence of $(\partial_+ \co(F_n))_{n \in \N^*}$ and $(\partial_+ \co(G_n))_{n \in \N^*}$ to $\partial_+ \co(F)$ and $\partial_+ \co(G)$, respectively.
	Again, as $f$ and $g$ are càdlàg, we have $\max ( \|f\|_\infty, \|g \|_\infty ) + 1 =: K < \infty$,
	and $\max ( \|f_n\|_\infty, \|g_n\|_\infty) \le K$ for all $n \in \N^*$, which yields by definition of the convex hull that $\co(F_n)(u) \ge F_n(0) - Ku$ and $\co(F_n)(u) \ge F_n(1) - K(1 - u)$.
	Hence,
	\[
		\frac{\co(F_n)(u) - \co(F_n(0))}{u} \ge -K, \quad
		\frac{\co(F_n)(1) - \co(F_n)(u)}{1 - u} \le K,
	\]
	and by monotonicity of the one-sided derivatives (and the same reasoning for $G_n$) we obtain that $\max( \|\partial_+ \co(F_n)\|_\infty, \|\partial_+ \co(G_n)\|_\infty) \le K$.
	Then dominated convergence yields that $(\partial_+ \co(F_n))_{n \in \N^*}$ and $(\partial_+ \co(G_n))_{n \in \N^*}$ converge in $L^p(\lambda)$ to $\partial_+ \co(F)$ and $\partial_+ \co(G)$, respectively.
	Finally, by applying \eqref{eq:convex_hull_contracting.assertion} and the triangle inequality we get the desired inequality:
	\begin{align*}
		\left\Vert\partial_+(\co(F)-\co(G))\right\Vert_p
		&= \lim_{n\to \infty} \left\Vert\partial_+(\co(F_n)-\co(G_n))\right\Vert_p
		\\
		&\le \lim_{n\to+\infty} \left\Vert f_n - g_n \right\Vert_p = \left\Vert f	 - g \right\Vert_p.
	\end{align*}
	
	It remains to show \eqref{eq:convex_hull_contracting.assertion} for piecewise constant, càdlàg functions $f$ and $g$ with finitely many jumps.
	To this end, let $(a_k)_{0\le k\le n}$ be a partition of $[0,1]$ adapted to $f$ and $g$, i.e., $0 = a_0 < \ldots < a_n = 1$ and for all $k\in\{0,\cdots,n-1\}$, $f\vert_{[a_k,a_{k+1})}$ and $g\vert_{[a_k,a_{k+1})}$ are constant. For $k\in\{0,\cdots,n\}$, we consider
	\begin{align*}
		F^k:x\mapsto&\begin{cases}\co(F\vert_{[0,a_k)})(x) &\text{if } x \in [0,a_k), \\ F(x) & \text{else;}\end{cases}
		\\
		G^k:x\mapsto &\begin{cases}
	\co(G\vert_{[0,a_k)})(x)&\text{if } x \in [0,a_k), \\ G(x) & \text{else,}
	\end{cases}
	\end{align*}
	and write $f^k = \partial_+ F^k$ and $g^k = \partial_+ G^k$.
	
	Note that $F^0 = F^1 = F$, $G^0 =G^1 = G$ and $F^n = \co(F)$, $G^n = \co(G)$.
	By induction we will show that, for $k \in \{0,\ldots n - 1\}$.
	\begin{equation}
		\label{eq:convex_hull_contracting.induction}
		\Vert f^{k+1} - g^{k+1} \Vert_p \leq \Vert f^k - g^k\Vert_p.
	\end{equation}
	As the initial case is trivial, we assume that \eqref{eq:convex_hull_contracting.induction} holds for $0\le k\leq n-2$.
	First, observe	
	\begin{align*}
		\nonumber
	\left\Vert f^{k+1} - g^{k+1}\right\Vert_p^p &= \left\Vert (f^{k+1} - g^{k+1})\vert_{[0,a_{k+1})}\right\Vert_p^p + \left\Vert (f^{k+1}-g^{k+1})\vert_{[a_{k+1},1]}\right\Vert_p^p
	\\
	&=\left\Vert \partial_+ \left( \co(F\vert_{[0,a_{k+1})})-\co(G\vert_{[0,a_{k+1})}) \right) \right\Vert_p^p + \left\Vert (f^k - g^k)\vert_{[a_{k+1},1]}\right\Vert_p^p.
	\end{align*}
	Applying Lemma \ref{lem:convex hull composition} with $a=a_k$, $b=a_{k+1}$ yields $\co(F\vert_{[0,a_{k+1})})=\co(F^k\vert_{[0,a_{k+1})})$ and $\co(G\vert_{[0,a_{k+1})})=\co(G^k\vert_{[0,a_{k+1})})$, so that,
	\[ \textstyle
		\left\Vert f^{k+1} - g^{k+1}\right\Vert_p^p=\left\Vert \partial_+ \left( \co(F^k\vert_{[0,a_{k+1})})-\co(G^k\vert_{[0,a_{k+1})}) \right) \right\Vert_p^p + \left\Vert (f^k - g^k)\vert_{[a_{k+1},1]}\right\Vert_p^p
	\]
	Since $f$ and $g$ are absolutely bounded by some constant $C > 0$, we have $\co(F\vert_{[0,a_{k})})(u) \ge F(a_k) - C(a_k -u  )$, $\co(G\vert_{[0,a_{k+1})})(u) \ge G(a_k) - C( a_k-u )$, thus, $\lim_{u \nearrow a_k}\co(F\vert_{[0,a_{k})})(u) = F(a_k)$ and $\lim_{u \nearrow a_k} \co(G\vert_{[0,a_{k})})(u) = G(a_k)$.
	In particular, $F^k$ and $G^k$ are continuous.
	We can apply Lemma \ref{lem:convex_affine} with $a=a_k$, $b=a_{k+1}$ % and the maps $F^k$, $G^k$
        to obtain
	\begin{align*}
		% \left\Vert \partial_+ (\co(F\vert_{[0,a_{k+1})}) - \co(G\vert_{[0,a_{k+1})})) \right\Vert_p &= 
		\left\Vert \partial_+(\co(F^k\vert_{[0,a_{k+1})}) - \co(G^k\vert_{[0,a_{k+1})}))\right\Vert_p
		&\le\left\Vert (f^k-g^k)\vert_{[0,a_{k+1})}\right\Vert_p,
	\end{align*}
	from which we deduce \eqref{eq:convex_hull_contracting.induction}.
	In particular, we have shown the assertion:
	\[
		\left\Vert \partial_+ (\co(F) - \co(G)) \right\Vert_p = \left\Vert f^n-g^n \right\Vert_p \le \left \Vert f^0 - g^0 \right \Vert_p =\left\Vert f - g \right\Vert_p.\qedhere
	\]
\end{proof}
The proof of Lemma \ref{lem:convex_affine} relies on the next two lemmas.
   \begin{lemma}
	\label{lem:convex_est}
	Let $\theta \colon \R \to \R$ be convex, $x,z \in \R$ with $x < z$, and $y, \hat y \in [x,z]$.
	Then
	\begin{equation}\textstyle
		\label{eq:convex_est.assertion}
		\theta(y) - \theta(x) \le \frac{y - x}{z - \hat y} \left( \theta(z) - \theta(\hat y) \right).
	\end{equation}
\end{lemma} \begin{proof}[Proof of Lemma \ref{lem:convex_est}]
	Since $\theta$ is convex, we have
	\[	\textstyle
		\theta(y) - \theta(x) \le \frac{y - x}{z - x} \left( \theta(z) - \theta(x) \right) 
		\mbox{ and }
		\frac{z - \hat y}{z - x} \left( \theta(z) - \theta(x) \right) \le \theta(z) - \theta(\hat y).
	\]
	Combining these two inequalities yields	\eqref{eq:convex_est.assertion}.
\end{proof}

\begin{lemma}\label{lem:quantile Function increasing ordering} 
	Let $\mu\in\mathcal P_1(\R)$ and $f : \R \to \R$ be a measurable map such that $\mu$ is equal to the image $f_\# \lambda$ of the Lebesgue measure $\lambda$ on $(0,1)$ by $f$.
    Then
	\[ \textstyle
	   \forall v\in [0,1], \quad \int_v^1f(u)\,du\le\int_v^1F_\mu^{-1}(u)\,du.\]
\end{lemma} \begin{proof}[Proof of Lemma \ref{lem:quantile Function increasing ordering}]	Let $v\in[0,1]$, $c:(x,y)\mapsto\mathbbm 1_{\{x<1-v\}}y$ and $U$ be a uniform random variable on $(0,1)$.
	Observe that for all $x,x',y,y'\in\R$ with $x\le x'$ and $y\le y'$ we have
	\[ \textstyle
	c(x',y')-c(x,y')-c(x',y)+c(x,y)\le 0.
	\]	
	Then by \cite[Theorem 3.1.2]{RaRu98} we have
	\[\textstyle
	\int_v^1f(u)\,du=\E[c(1-U,f(U))]\le\E[c(1-U,F_\mu^{-1}(U))]=\int_v^1F_\mu^{-1}(u)\,du.\qedhere
	\]	
\end{proof} \begin{proof}[Proof of Lemma \ref{lem:convex_affine}]
	Let $f, g \colon [0,b) \to \R$ be given by the right-hand derivative of $F$ and $G$ resp., that is $f := \partial_+ F$, $g := \partial_+ G$.
	% Observe that
	% \begin{equation}
	% 	\label{eq:convex_affine.f_prop}
	% 	F(x) = F(0) + \int_0^x f(u)	\, du \mbox{ and } G(x) = G(0)+\int_0^x g(u) \, du \mbox{ for }x \in [0,b].
	% \end{equation}
	Our first goal is to find an explicit representation of the convex hulls of $F$ and $G$.
	To this end, consider the infimum
	\begin{equation}\textstyle
		\label{eq:convex_affine.c_def}
		c := \inf \left\{ x\in[0,b) \mid f(x) \ge \frac{F(b)-F(x)}{b-x} \right\},
	\end{equation}
	which is well-defined and not greater than $a$ as $f(a)=\frac{F(b)-F(a)}{b-a}$.
	Moreover, the infimum in \eqref{eq:convex_affine.c_def} is attained by continuity of $F$ and right-continuity of $f$.
	When $c>0$, we get by \eqref{eq:convex_affine.c_def} and monotonicity of $f$ that $\sup_{x\in[0,c)}f(x)=\lim_{x \nearrow c} f(x) \le \lim_{x \nearrow c} \frac{F(b) - F(x)}{b - x}=\frac{F(b) - F(c)}{ b - c}$.
	Under the convention $\sup_{x\in[0,c)}f(x)=-\infty$ when $c=0$, we therefore find
	\begin{equation}\textstyle
		\label{eq:convex_affine.c_left_limit}
		\sup_{x\in[0,c)}f(x)\le \frac{F(b) - F(c)}{ b - c} =: \phi \le f(c).
	\end{equation}
	For $x \in [c,a]$, using \eqref{eq:convex_affine.c_left_limit} and the fact that $f$ is non-decreasing on $[0,a)$, we obtain 
	\begin{align}
		\nonumber
		\text{$\textstyle F(b) - F(x)$} &\text{$\textstyle=  F(b) - F(c) + F(c) - F(x) = (b - c)\phi - \int_c^x f(u) \, du$}
		\\
		\label{eq:convex_affine.c_computation}
		&\text{$\textstyle \le (b - c) \phi - (c - x) \phi = (b - x) \phi \le (b - x) f(x).$}
	\end{align}
	We claim that
	\begin{equation}
		\label{eq:convex_affine.co_pepresentation}
		\co(F)(x) = F(\min(x, c)) + \max(x - c, 0) \phi, \quad x \in [0,b]. 
	\end{equation}
	Denote the right-hand side of \eqref{eq:convex_affine.co_pepresentation} by $\tilde F$.
	Note that $\tilde F$ is convex since the right-hand derivative $\partial_+ \tilde F(x) = \mathbbm 1_{[0,c)} f(x) + \mathbbm 1_{[c,b)}\phi$ is non-decreasing by \eqref{eq:convex_affine.c_left_limit}.
	We calculate
	\[\textstyle
		F(x) = F(c) + F(x) - F(b) + F(b) - F(c) \ge F(c) - (b - x) \phi + (b - c) \phi= \tilde F(x),
	\]
	for $x \in [c,a]$, which yields that $\tilde F \le F$ on $[0,a] \cup \{ b \}$.
	Since both functions are affine on $[a,b]$ we conclude that $\tilde F \le F$ and therefore, by definition of the convex hull, $\tilde F \le \co(F)$.

	In order to show \eqref{eq:convex_affine.co_pepresentation}, it remains to verify that $\co(F)\le\tilde F  $.
	By convexity of $\co(F)$ and the inequality $\co(F)\le F$, we have, for $x \in [c,b]$,
	\[ \textstyle
		\co(F)(x) \le \frac{x - c}{b - c} F(b) + \frac{b - x}{b - c} F(c)= F(c) + (x - c) \phi  = \tilde F(x),
	\]
	and $\co(F)(x) \le F(x) = \tilde F(x)$ for $x \in [0,c]$, thus, $\co(F) \le \tilde F$.

	Reasoning the same way for $G$, we deduce that $d$ defined analogously to \eqref{eq:convex_affine.c_def}
	\[
		\textstyle d:= \inf \left\{ x \in [0,b) \colon g(x) \ge \frac{G(b) - G(x)}{b-x} \right\} 
	\]
	is not greater than $a$ and	has the properties
	\begin{gather}\textstyle
		\label{eq:convex_affine.d_left_limit}
		\sup_{x \in [0,d)} g(x) \le \frac{G(b) - G(d)}{b - d} =: \gamma \le g(d) % = \min_{x \in [d,a]} g(x)
                ,
		\\
		\label{eq:convex_affine.co_pepresentation2}
		\co(G)(x) = G(\min(x,d)) + \max(x - d, 0) \gamma.
	\end{gather}
	
	After this preparatory work we proceed to show the assertion, that is \eqref{eq:convex_affine.assertion}.
	Without loss of generality, we assume that $c \le d$. 
	Note that, by \eqref{eq:convex_affine.co_pepresentation} and \eqref{eq:convex_affine.c_left_limit},  $\partial_+ \co(F)(x) = \min( f(\min(x,c)), \phi)$ and, by \eqref{eq:convex_affine.co_pepresentation2} and \eqref{eq:convex_affine.d_left_limit}, $\partial_+ \co(G)(x) = \min( g(\min(x,d)), \gamma)$.
	Therefore, the left-hand side of \eqref{eq:convex_affine.assertion} coincides with
	\[\textstyle
		\int_0^c \theta \left( |f(u) - g(u)| \right) \, du + \int_c^d \theta \left(|\phi - g(u)| \right) \, du
		+ \int_d^b \theta \left( |\phi - \gamma| \right) \, du.
	\]
	We have to show
	\begin{equation}\textstyle
		\label{eq:convex_affine.toshow}
		\int_c^d\theta\left(|\phi - g(u)| \right) + \int_d^b \theta \left( |\phi - \gamma| \right) \, du \le
		\int_c^b \theta\left( |f(u) - g(u)| \right) \, du.
	\end{equation}

	\emph{Case 1}: Suppose that $\phi \ge \gamma$.
	By \eqref{eq:convex_affine.c_left_limit}, the monotonicity of $f$ on $[0,a)$ and \eqref{eq:convex_affine.d_left_limit},
	\[
		g(u) \le \gamma	\le \phi \le f(u)\quad u \in [c,d).
	\]
	Then, by applying Lemma \ref{lem:convex_est} (where in the notation of this lemma $x = \phi - \gamma$, $y = f(u) - \gamma$, $\hat y = \phi - g(u)$, $z = f(u) -  g(u)$), we find
	\begin{equation*}
		\label{eq:convex_affine.step1_1}
		\theta \left( f(u) - g(u) \right) - \theta \left( \phi - g(u) \right)\ge\theta \left( f(u) - \gamma \right) - \theta \left( \phi - \gamma \right),\quad u \in [c,d)
	\end{equation*}
\begin{equation*}
		\label{eq:convex_affine.step1_3}
		\mbox{so that }\int_c^d \theta \left( |f(u) - g(u) | \right) - \theta\left( \phi - g(u) \right) \, du 
		\ge
		\int_c^d \theta \left( f(u) - \gamma \right) - \theta(\phi - \gamma) \, du.
	\end{equation*}	Denoting by $\tilde \phi$ the fraction $\frac{F(b) - F(d)}{b - d}$, we have, since $\theta$ is increasing and convex, 
	\begin{equation*}\textstyle
		\label{eq:convex_affine.step1_2}
		\int_d^b \theta\left( |f(u) - g(u)| \right)	\, \frac{du}{b - d} \ge \theta \left(\left| \int_d^b \frac{f(u) - g(u)}{b - d} \, du \right| \right) = \theta\left(|\tilde \phi - \gamma|\right).
	\end{equation*}
Summing the two last inequalities, then using the convexity of $\theta$, we deduce that
	\begin{multline}\textstyle
		\label{eq:convex_affine.step1_4}
		\int_c^b \theta\left( |f(u) - g(u)| \right) \, du - \int_c^d  \theta(\phi - g(u)) \, du  - (b - d) \theta(\phi - \gamma)
		\\\textstyle
		\ge
		\int_c^d \theta(f(u) - \gamma) - \theta(\phi - \gamma) \, du + \int_d^b \theta(|\tilde \phi - \gamma|) - \theta(\phi - \gamma) \, du
		\\\textstyle
		\ge
		(b - c) \left( \theta \left( \frac{1}{b - c} \left( \int_c^d f(u) - \gamma \, du + \int_d^b |\tilde \phi - \gamma|du \right) \right) - \theta(\phi - \gamma) \right),
	\end{multline}
	Since $\theta$ is non-decreasing $\frac{1}{b - c}\left(\int_c^d f(u) - \gamma \, du + \int_d^b \tilde \phi - \gamma \, du\right)= \phi - \gamma$,
	we find that \eqref{eq:convex_affine.step1_4} is non-negative, from which we derive \eqref{eq:convex_affine.toshow}.

	\emph{Case 2}: Suppose that $\phi < \gamma$ and let $e := \inf\{u\in[c,d]\mid g(u)\ge\phi\}$ where by convention the infimum over the empty set is defined as $d$.
	By \eqref{eq:convex_affine.c_left_limit} and the monotonicity of $f$ on $[0,a)$, we have $g(u) \le \phi \le f(u)$ for $u \in [c,e)$, thus,
	\begin{equation}\textstyle
		\label{eq:convex_affine.step2_1}
		\int_c^e \theta(\phi - g(u)) \, du \le \int_c^e \theta(f(u) - g(u)) \, du.
	\end{equation}
	On the other hand, by \eqref{eq:convex_affine.co_pepresentation},
\begin{equation}\textstyle
   \forall x\in[c,b],\;F(x)\ge\co(F)(x)=F(b)+(x-b)\phi\label{eq:minoFcb}
\end{equation}
 so that $\tilde \phi = \frac{F(b) - F(d)}{b - d} \le \phi$.
	As $\theta$ is increasing and convex, we obtain
	\begin{equation}\textstyle
		\label{eq:convex_affine.step2_2}
		\int_d^b \theta\left( |f(u) - g(u)| \right) \, \frac{du}{b - d} \ge 
		\theta \left( \gamma - \tilde \phi \right).
	\end{equation}
	As consequence of \eqref{eq:convex_affine.step2_1} and \eqref{eq:convex_affine.step2_2}, the following inequality suffices to get \eqref{eq:convex_affine.toshow}:
	\begin{equation}\textstyle
		\label{eq:convex_affine.step2_3}
		\int_e^d \theta(g(u) - \phi) \, du + (b - d) \theta(\gamma - \phi ) \le \int_e^d \theta \left(|f(u) - g(u)|\right) \, du + (b - d) \theta(\gamma - \tilde \phi).
	\end{equation}
	Showing \eqref{eq:convex_affine.step2_3} is equivalent to proving that the respective images $\mu$ and $\nu$ of the Lebesgue measure $\lambda$ on $(0,1)$ by the non-decreasing maps
	\begin{align*}
		T^1(u) :=& 
		\begin{cases}
			g( e + (b - e) u ) - \phi & u < \frac{d - e}{b - e},
			\\
			\gamma - \phi & \text{else},
		\end{cases}
		\\
		T^2(u) :=&
		\begin{cases}
			|f(e + (b - e)u) - g(e + (b - e)u )| & u < \frac{d - e}{b - e},
			\\
			\gamma - \tilde \phi & \text{else}
		\end{cases}
	\end{align*}
	 are in the increasing convex order ($\le_{icx}$).
	By \cite[Theorem 4.A.3]{ShSh07}, this is equivalent to
	\begin{equation}\textstyle
		\label{eq:convex_affine.step2_4}
		\int_v^1 F_\mu^{-1}(u) \, du \leq \int_v^1 F_\nu^{-1}(u)\, du,\quad v \in [0,1].
	\end{equation}
	Since $T^1$ is non-decreasing, we have by \cite[Lemma A.3]{AlCoJo20} that $T^1(u) = F_\mu^{-1}(u)$ for $\lambda$-almost every $u \in (0,1)$.
	This observation combined with Lemma \ref{lem:quantile Function increasing ordering} leads to
	\begin{equation*}\textstyle
		\label{eq:convex_affine.step2_5}
		\int_v^1 F_\mu^{-1}(u) \, du = \int_v^1 T^1(u) \, du \mbox{ and }\int_v^1 F_\nu^{-1}(u) \, du \ge \int_v^1 T^2(u) \, du,\; v \in [0,1]
    \end{equation*}
    and \eqref{eq:convex_affine.step2_4} is implied by
    \begin{equation}\textstyle
        \int_v^1 T^1(u) \, du\le \int_v^1 T^2(u) \, du,\; v \in [0,1].
    \end{equation}              
	Recall that $\tilde \phi \le \phi$, so that this inequality holds for $v \in [\frac{d - e}{b - e},1]$.
	Next, abbreviate $\frac{d - e}{b - e} =: w$, let $v \in [0,w)$, and write $\hat e := e + (b - e)v$.
	We have by the triangle inequality that
	\begin{align*}
		\text{$\textstyle \int_v^1 T^1(u) \, du$}
		&\text{$\textstyle \le \int_v^w f(e + (b - e)u) + T^2(u) - \phi \, du + \int_w^1 \gamma - \phi \, du$}
		\\
		&\text{$\textstyle =\frac{F(d) - F(\hat e) - (d - \hat e) \phi + (b - d)(\gamma - \phi)}{b - e} + \int_v^w T^2(u) \, du.$}
	\end{align*}
	Remember that $(b - d) \tilde \phi = F(b) - F(d)$ and \eqref{eq:minoFcb} applies to $x=\hat e$ since $\hat e\ge e\ge c$.
	Thus,
	\[\textstyle	
		F(d) - F(\hat e) = F(b) - F(\hat e) - (b - d) \tilde \phi \le (b- \hat e) \phi - (b - d) \tilde \phi.
	\]
	We obtain
	\begin{align*}
		\text{$\textstyle \int_v^1 T^1(u) \, du$}&\text{$\textstyle
		\le \frac{(b - d)(\phi - \tilde \phi + \gamma - \phi)}{b - e} + \int_v^w T^2(u) \, du$}
		\\ \textstyle
		&\text{$\textstyle \le \frac{b - d}{b - e} (\gamma - \tilde \phi) + \int_v^w T^2(u) \, du = \int_v^1 T^2(u) \, du.$}\qedhere
	\end{align*}
\end{proof}
\bibliographystyle{abbrv}
\bibliography{joint_biblio}
\end{document}